\documentclass[12pt]{amsart}

\usepackage{amsfonts,latexsym,amsmath,amsthm,amssymb, color}
\usepackage[margin=1.0in]{geometry}

\def\H{{\mathcal H}}

\def\m{{\mathbb M}}

\def\R{{\mathbb R}}

\newcommand{\Ker}{\mathsf{Ker}~ }

\newcommand{\Ran}{\mathsf{Ran}~ }

\newtheorem*{theorem*}{Theorem}
\numberwithin{equation}{section}

\newtheorem{theorem}{Theorem}[section]

\newtheorem{lemma}[theorem]{Lemma}

\begin{document}

\title[Perturbation theory for the matrix...]{Perturbation theory for the matrix square root and matrix modulus}
\author{Marcus Carlsson}
\address{Centre for Mathematical Sciences, Lund University\\Box 118, SE-22100, Lund,  Sweden\\}
\email{marcus.carlsson@math.lu.se}

\begin{abstract}
We provide first order perturbation formulas for the matrix square root (in the positive semi-definite case) and the matrix modulus (in the general case). The results are new for singular matrices, and extend previously known Fr\'{e}chet differentiability formulas provided by the Daleckii-Krein theorem.
\end{abstract}
\maketitle


\section{Introduction}

Let $\m_n$ denote the set of $n\times n $ complex matrices, let $\H_n$ denote the subset of Hermitian self-adjoint matrices and $\H_n^+$ the subset of positive semi-definite matrices. By the spectral theorem, each matrix $A\in\H_n$ can be decomposed as $A=U\Lambda_\alpha U^*$ where $\alpha$ are the eigenvalues, $\Lambda_{\alpha}$ the diagonal matrix with the vector $\alpha$ on the diagonal, and $U$ is a unitary matrix whose columns are the eigenvectors. For $A\in\H_n^+$, the matrix square root is then defined as \begin{equation}\label{spec_dec}\sqrt{A}=U\Lambda_{\sqrt{\alpha}} U^*,\quad A\in\H_n^+\end{equation}
where the square root is taken pointwise on $\alpha$. The matrix modulus is in turn defined as $$|A|=\sqrt{A^*A},\quad A\in\m_n.$$ In this article we study the perturbation theory for these transformations, i.e.~we are interested in leading order terms when $A$ is replaced by $A+E$ and $E$ is small.

If $A$ is nonsingular (i.e.~invertible), the first order perturbation of both operations can be completely characterized by the so called Daleckii-Krein theorem, showing that the two maps are in fact Fr\'{e}chet differentiable. To be more precise, if we introduce the matrix $[\sqrt{\cdot},\alpha]$ by setting $$[\sqrt{\cdot},\alpha]_{i,j}=\frac{1}{\sqrt{\alpha_i}+\sqrt{\alpha_j}}$$ and the notation $\hat E=U^*EU$ (where $A=U\Lambda_\alpha U^*$) then it holds that \begin{equation}\label{dk}\sqrt{A+E}=\sqrt{A}+U([\sqrt{\cdot},\alpha]\circ \hat E)U^*+O(\|E\|^2),\end{equation}
where $\circ$ denotes Hadamard multiplication. Note that the middle term is linear in $E$ and hence it equals the Fr\'{e}chet derivative of $\sqrt{A}$. We revisit the Daleckii-Krein theorem in Section \ref{secdk}. 

If $A$ is singular on the other hand, then the non-differentiability of $t\mapsto\sqrt{t}$ at $t=0$ clearly prohibits any Fr\'{e}chet differentiability of $\sqrt{A}$. Despite this, we derive a first order perturbation result for $\sqrt{A+E}$ similar to \eqref{dk}, where the ordo term is $O(\|E\|^{3/2})$. If $A$ has a non-trivial kernel, the term corresponding to $U([\sqrt{\cdot},\alpha]\circ \hat E)U^*$ is no longer linear, but can be explicitly computed using a Schur complement of $E$ in the basis provided by $U$. Hence it provides a concrete, albeit non-linear, formula for the first order perturbation. As it turns out, the theory holds for $A\mapsto A^{s}$ for all $s>1/3$ and provides a neat extension of the Daleckii-Krein theorem to such functions. More precisely, we show that $$(A+E)^s=A^s+L(E)+O(\|E\|^r)$$ where $r>1$ and the term $L(E)$ is not linear (unless $s\geq 1$ or $\Ker A=\{0\}$), but the non-linearity is confined to computing $D^s$ where $D$ is a particular Schur complement of $E$, we refer to Section \ref{sec_pert} for the details.

In Section \ref{sec3} this result is then applied to study of the perturbation of the matrix modulus, i.e.~the function $E\mapsto |A+E|$, where we lift the assumption that $A$ and $E$ are Hermitian. Again, we provide an elementary formula for its first order approximation with an error term of $O(\|E\|^{3/2})$. We remark that if we let $P_0$ be the projection onto the kernel of $A$, then the only non-linear part of the first order perturbation term is $|P_0 E P_0|$, and hence the expression is again readily computable.

The study of the matrix square root began with Cayley in 1858 \cite{cayley1858ii}, and in the light of its tremendous influence on mathematical analysis, numerical analysis as well as mathematical physics, it is rather surprising that its perturbation theory is not yet fully understood. We refer to Ch.~6-8 \cite{higham2008functions} for modern uses of matrix powers and matrix-modulus in numerical analysis, whereas a recent account of what is known on the theoretical side is found in \cite{bhatia2013matrix}, in particular Ch.~X.

\section{The Daleckii-Krein theorem}\label{secdk}

Given $A=U\Lambda_\alpha U^*\in\H_n$ and a perturbation $A+E\in \H_n$, we recall that the matrix $\hat E$ appearing in \eqref{dk} is defined as $$\hat E=U^*EU,$$ leaving the dependence on $A$ implicit. Given a real valued function $f$ some subset of $\R$ containing the spectrum of $A$, the standard functional calculus is defined as $$f(A)=U\Lambda_{f(\alpha)}U^*.$$ If $f$ is also a $C^1$ function, we introduce the matrix
\begin{equation}\label{flambda}[f,\alpha]({i,j})=
\left\{\begin{array}{cc}
  \frac{f(\alpha_i)-f(\alpha_j)}{\alpha_i-\alpha_j} & \alpha_i\neq \alpha_j \\
  \partial_i{f}(\alpha_i) & \alpha_i=\alpha_j
\end{array}\right.
\end{equation}

The version of the Daleckii-Krein theorem we need is a bit stronger than what is usually found in books, and reads as follows.

\begin{theorem}\label{thmdk}
Let $A=U\Lambda_\alpha U^*\in\H_n$ be given and let $f$ be $C^2$ in a neighborhood of $\alpha$. Then $$f(A+E)=f(A)+ U\Big([f,\alpha]\circ \hat E\Big)U^*+O(\|E\|^2)$$
where the size of the error is locally independent of $A$. In other words, there exists a constant $C>0$ such that
\begin{equation}\label{gtyh}\|f(A+E)-f(A)- U\Big([f,\alpha]\circ \hat E\Big)U^*\|\leq C\|E\|^2\end{equation} for all $E$ in a neighborhood of 0, which is also invariant for small perturbations of $A$ itself.
\end{theorem}

The above statement can actually be deduced from the original article \cite{daletskii1965integration} with a bit of detective work, (in fact, it is curious to note that Daleckii and Krein do not even mention that they have shown Fr\'{e}chet differentiability). More precisely, formula (2.13) of \cite{daletskii1965integration} with $\varepsilon=1$, $n=1$, $H_1=E$ shows that $f(A+E)=f(A)+U\Big([f,\alpha]\circ \hat E\Big)U^*+R_1(1)$ where $R_1(1)$ is an error term, (although the notation in \cite{daletskii1965integration} looks rather different). The error term involves an iterated operator integral, which has a concrete estimate in (2.16), but a closer inspection of this estimate shows that we need to assume $f\in C^4$ to deduce \eqref{gtyh}. However, this is due to the fact that Daleckii and Krein work with operators on infinite dimensional spaces and, as observed in the introduction of \cite{daletskii1965integration}, the operator integrals reduce to finite sums in the matrix setting. With this in mind, the estimate (1.3) in \cite{daletskii1965integration} can be applied to the estimate (2.15) of $R_1(1)$, and \eqref{gtyh} follows. For this to work out, it is necessary that $f$ is $C^2$ in an interval $(a,b)$ including $\alpha$, but if this is not already the case, it can always be achieved by basic interpolation techniques, we omit the details.

If one only assumes $f\in C^1$, then Fr\'{e}chet differentiability appears e.g.~in \cite{bhatia2013matrix} (Theorem V.3.3) but there seems to be no easy way to obtain the more precise control on the remainder given in the above theorem. A basic proof of \eqref{thmdk} will appear in \cite{carlsson2018perturbation3}, and it can also be deduced from Theorem 6.9 in \cite{sendov2007higher} and a bit of work. It is interesting to note that if $f(t)=\sum_{j=0}^\infty a_jt^j$ is entire, the first order perturbation can also be computed upon expanding the terms in $f(A+E)=\sum_{j=0}^\infty a_j(A+E)^j$ and collecting those that only involve one $E$, but the resulting expression is quite different than \eqref{gtyh}. Estimation of higher order terms in this case is a straightforward matter, see e.g.~\cite{mathias1993approximation}.

\section{Perturbation of powers}\label{sec_pert}

The identity \eqref{dk} for $\sqrt{A+E}$ is an immediate application of Theorem \ref{thmdk} to $f(t)=\sqrt{t}$, and applies whenever $A\in\H_n^+$ has no kernel. On the other hand, if $A$ has a kernel it is definitely not Fr\'{e}chet differentiable in the classical sense. Despite this, we shall find a tractable formula for $o(\|E\|)$ approximation of $(A+E)^{1/p}$ for $1<p<3$. Our analysis relies on finer results concerning perturbation of eigenvalues and eigenvectors, recently shown in the adjacent paper \cite{carlsson2018perturbation1}.

Note that $(A+E)^{1/p}=U(\Lambda_\alpha+\hat E)^{1/p}U^*$ so we shall primarily work with $(\Lambda_\alpha+\hat E)^{1/p}$. Furthermore we decompose the space $\H_n$ into $\Ran \Lambda_\alpha\oplus \Ker \Lambda_\alpha$ which results in the representation of $\Lambda_\alpha$ via $\left(
                                                                 \begin{array}{cc}
                                                                   \Lambda_{\alpha^+} & 0 \\
                                                                   0 & 0 \\
                                                                 \end{array}
                                                               \right)$, where $\Lambda_{\alpha^+}$ is diagonal containing the non-zero eigenvalues $\alpha^+$. This should be compared with the decomposition (4.3) in \cite{carlsson2018perturbation1}. Note that we have chosen to flip the order of $\Lambda_{\alpha^+}$ and $0$ since it is clearly more natural for the present situation. As in (4.3) of \cite{carlsson2018perturbation1} we decompose $\hat E$ using a Schur complement to express the fourth quadrant;  \begin{equation}\label{spec_dec2}\Lambda_{\alpha}+\hat E=\left(
                                                                 \begin{array}{cc}
                                                                   \Lambda_{\alpha^+} & 0 \\
                                                                   0 & 0 \\
                                                                 \end{array}
                                                               \right)+\left(
                                                                 \begin{array}{cc}
                                                                   B & C \\
                                                                   C^* & C^*(\Lambda_{\alpha^+}+B)^{-1}C+D \\
                                                                 \end{array}
                                                               \right),\end{equation}
A nice feature with Schur complements is that  $\Lambda_\alpha+\hat E$ is positive semi-definite if and only if $D$ is (as long as $\Lambda_{\alpha^+}+B\geq 0$), see e.g.~\cite{boyd2004convex}. Since we are only interested in perturbations that stay in $\H_n^+$ we will assume that $D\geq 0$ from the outset.

Given $s>0$ we define $[(\cdot)^{s},\alpha]_1$ as in \eqref{flambda} except that we set it to 1 when $\alpha_{i}=\alpha_{j}=0$, i.e.
\begin{equation}\label{flambda2}([(\cdot)^{s},\alpha]_1)_{i,j}=
\left\{\begin{array}{cc}
  \frac{\alpha_{i}^{s}-\alpha_{j}^{s}}{\alpha_{i}-\alpha_{j}} & \alpha_{i}\neq \alpha_{j} \\
  s\alpha_{i}^{s-1} & \alpha_{i}=\alpha_{j}>0\\
  1 & \alpha_{i}=\alpha_{j}=0
\end{array}\right.
\end{equation}
Note that for $s>1$, $[(\cdot)^{s},\alpha]$ is defined also via \eqref{flambda}, and that $[(\cdot)^{s},\alpha]_1$ differs from $[(\cdot)^{s},\alpha]_1$ on indices such that $\alpha_i=\alpha_j=0$. The main result of this section reads as follows.

\begin{theorem}\label{thm_frac_pow}
Let $1<p<3$ and set $r=\min(1+1/p,3/p)$. Suppose that both $A$ and its perturbation $A+E$ are positive semi-definite. Then, in terms of the decomposition \eqref{spec_dec2} of $E$, we have
\begin{equation}\label{lo}(A+E)^{1/p}=A^{1/p}+U\left([(\cdot)^{1/p},\alpha]_1\circ \left(
                                                                 \begin{array}{cc}
                                                                   B & C \\
                                                                   C^* & D^{1/p} \\
                                                                 \end{array}
                                                               \right)\right)U^*+O(\|E\|^r)\end{equation} where the size of the error is locally independent of $\alpha^+$.
\end{theorem}

We begin with some preliminary results. We will frequently need the following inequality \begin{equation}\label{lipfrac}\|Z^{1/p}-{A}^{1/p}\|_2\leq n^{\frac{p-1}{2}}{\|Z-A\|_2}^{1/p},\end{equation}
which surprisingly seems to have been established only recently by T. P. Wihler \cite{wihler2009holder}. We will refer to this as Wihler's inequality.

Next we need three lemmas concerning perturbation of projection operators onto the kernel and cokernel of $A$. Let $\Gamma_1$ be a circle enclosing the non-zero spectrum of $A$, but not 0, and $\Gamma_0$ a circle doing the opposite. The proof utilizes the two projections
\begin{equation}\label{prj}P_i(\tilde E)=\int_{\Gamma_i}(zI-(\Lambda_{\alpha}+\tilde E))^{-1}\frac{dz}{2\pi i},\quad i=0,1\end{equation} where we use $\tilde{E}$ to not confuse this perturbation with $E$ or $\hat E$.
\begin{lemma}\label{gt} Given $\tilde E\in\H_n$ and decomposition $\tilde E=\left(
                                                                 \begin{array}{cc}
                                                                   \tilde B & \tilde C \\
                                                                   \tilde C^* & \tilde D \\
                                                                 \end{array}
                                                               \right)$  we have $$P_1(\tilde E)=\left(
                                                                 \begin{array}{cc}
                                                                   I & \Lambda_{\alpha^+}^{-1}\tilde C \\
                                                                   \tilde C^* \Lambda_{\alpha^+}^{-1} & 0 \\
                                                                 \end{array}
                                                               \right)+O(\|E\|^2),\quad P_0(\tilde E)=\left(
                                                                 \begin{array}{cc}
                                                                   0 & -\Lambda_{\alpha^+}^{-1}\tilde C \\
                                                                   -\tilde C^* \Lambda_{\alpha^+}^{-1} & I \\
                                                                 \end{array}
                                                               \right)+O(\|E\|^2)$$ and the size of the errors are locally independent of $\alpha^+$.
\end{lemma}
\begin{proof}
This follows by Theorem \ref{thmdk} applied to a $C^2$ function which is 0 at 0 and 1 in a neighborhood of $\alpha^+$, or vice versa. It can also be deduced directly from \eqref{prj} using basic residue calculus.

\end{proof}

\begin{lemma}\label{gt1} Let $p,r$ be as in Theorem \ref{thm_frac_pow}. Let $Z\in\H^+_n$ be such that $\Ker Z\perp \Ker A$, and let $\tilde E$ be as in the previous lemma. Let $b>0$ be an upper bound for both $\|\tilde E\|$ and $\|Z\|$. Then $$\left(P_0(\tilde E)ZP_0(\tilde E)\right)^{1/p}-Z^{1/p}=O(b^r),$$
and the size of the errors are locally independent of $\alpha^+$.
\end{lemma}
\textit{Remark:} Applying Wihler's inequality \eqref{lipfrac} to $Z$ and $P_0(\tilde E)ZP_0(\tilde E)$ we achieve an error of $O(b)$. This shows that this inequality, while being sharp in general, can be improved for perturbations with a special structure.

\begin{proof}
As usual we choose a basis such that $A=\left(
                                                                 \begin{array}{cc}
                                                                   \Lambda_{\alpha^+} & 0 \\
                                                                   0 & 0 \\
                                                                 \end{array}
                                                               \right)$ and we can pick the vectors in $\Ker A$ such that also $Z$ diagonalizes, i.e. $Z=\left(
                                                                 \begin{array}{cc}
                                                                   0 & 0 \\
                                                                   0 & \Lambda_\zeta  \\
                                                                 \end{array}
                                                               \right)$. By Lemma \ref{gt1} we have
\begin{align*}&P_0(\tilde E)Z P_0(\tilde E)=
\left(
\begin{array}{cc}
0 & -\Lambda_{\alpha^+}^{-1}\tilde C \Lambda_\zeta  \\
- \Lambda_\zeta  \tilde C^* \Lambda_{\alpha^+}^{-1} & \Lambda_\zeta  \\
\end{array}
\right)+O(b^3)=\\&\left(
\begin{array}{c}
 -\Lambda_{\alpha^+}^{-1}\tilde C  \\
 I \\
\end{array}
\right)\Lambda_\zeta \left(
\begin{array}{c}
-\Lambda_{\alpha^+}^{-1}\tilde C  \\
 I \\
\end{array}
\right)^*+O(b^3),
\end{align*}
which is almost a spectral decomposition of $P_0(\tilde E)Z P_0(\tilde E)$, except for the the $O(b^3)$-term and the fact that $\left(
\begin{array}{c}
-\Lambda_{\alpha^+}^{-1}\tilde C  \\
 I \\
\end{array}
\right)$ is not orthogonal. However, using induction and the Gram-Schmidt orthogonalization process, it is easy to see that there exists a matrix $W$ with $W^*W=I$ and
$$\left(
\begin{array}{c}
-\Lambda_{\alpha^+}^{-1}\tilde C  \\
 I \\
\end{array}
\right)=W+O(b^2),$$ where the error is locally independent of $\alpha^+$. By the above identities and Wihler's inequality \eqref{lipfrac}, it follows that
\begin{align*}   &\left(P_0(\tilde E)Z P_0(\tilde E)\right)^{1/p}=\left(W\Lambda_\zeta W^*+O(b^3)\right)^{1/p}=W\Lambda_\zeta ^{1/p}W^*+O(b^{3/p})=\\&
\left(
\begin{array}{c}
-\Lambda_{\alpha^+}^{-1}\tilde C  \\
 I \\
\end{array}
\right)\Lambda_\zeta ^{1/p}\left(
\begin{array}{c}
-\Lambda_{\alpha^+}^{-1}\tilde C  \\
 I \\
\end{array}
\right)^*+O(b^2 b^{1/p})+O(b^{3/p})=\\&\left(
\begin{array}{cc}
0 & -\Lambda_{\alpha^+}^{-1}\tilde C \Lambda_\zeta ^{1/p} \\
- \Lambda_\zeta ^{1/p} \tilde C^* \Lambda_{\alpha^+}^{-1} & \Lambda_\zeta ^{1/p} \\
\end{array}
\right)+O(b^{3/p})=\\&\left(
\begin{array}{cc}
0 & 0 \\
0 & \Lambda_\zeta ^{1/p} \\
\end{array}
\right)+\left(
\begin{array}{cc}
0 & -\Lambda_{\alpha^+}^{-1}\tilde C \Lambda_\zeta ^{1/p} \\
- \Lambda_\zeta ^{1/p} \tilde C^* \Lambda_{\alpha^+}^{-1} & 0 \\
\end{array}
\right)+O(b^{3/p}).
\end{align*}
The first matrix equals $Z^{1/p}$, the second matrix is $O(b^{1+1/p})$, and the lemma follows.
\end{proof}

\begin{lemma}\label{gt2} In the setting of the previous lemma, we also have $$\left(P_1(\tilde E)ZP_0(\tilde E)\right)^{1/p}=O(b^2),\quad \left(P_1(\tilde E)ZP_1(\tilde E)\right)^{1/p}=O(b^3)$$
and the size of the errors are locally independent of $\alpha^+$.
\end{lemma}
\begin{proof}
By Lemma \ref{gt} and analogous computations as in the previous proof we get $$P_1(\tilde E)Z P_0(\tilde E)=
\left(
\begin{array}{c}
  \Lambda_{\alpha^+}^{-1}\tilde C  \\
0 \\
\end{array}
\right)\Lambda_\zeta \left(
\begin{array}{c}
-\Lambda_{\alpha^+}^{-1}\tilde C  \\
 I \\
\end{array}
\right)^*+O(b^3)=O(b^2),
$$
and $$P_1(\tilde E)Z P_1(\tilde E)=
\left(
\begin{array}{c}
  \Lambda_{\alpha^+}^{-1}\tilde C  \\
 0\\
\end{array}
\right)\Lambda_\zeta \left(
\begin{array}{c}
  \Lambda_{\alpha^+}^{-1}\tilde C  \\
 0 \\
\end{array}
\right)^*+O(b^3)=O(b^3),
$$
\end{proof}

With these observations we are ready for the proof of Theorem \ref{thm_frac_pow}. We satisfy with showing the result in the case $A=\Lambda_\alpha$ and $U=I$, since the general case follows easily from this, as noted initially.

\begin{proof}[Proof of Theorem \ref{thm_frac_pow}]
Note that \begin{equation}\label{tgf}\left(
                                                                 \begin{array}{cc}
                                                                   \Lambda_{\alpha^+}+B & C \\
                                                                   C^* & C^*(\Lambda_{\alpha^+}+B)^{-1}C \\
                                                                 \end{array}
                                                               \right)\left(
                                                                 \begin{array}{c}
                                                                   (\Lambda_{\alpha^+}+B)^{-1}C \\
                                                                   -I \\
                                                                 \end{array}
                                                               \right)=\left(
                                                                 \begin{array}{c}
                                                                   0 \\
                                                                   0 \\
                                                                 \end{array}
                                                               \right)\end{equation}
so the former matrix has rank $l=\mathsf{rank}(A)$ (as long as $\|E\|$ is sufficiently small). It follows that the shape of $t^{1/p}$ near 0 is irrelevant for \begin{equation}\label{lo9} \left(
                                                                 \begin{array}{cc}
                                                                   \Lambda_{\alpha^+}+B & C \\
                                                                   C^* & C^*(\Lambda_{\alpha^+}+B)^{-1}C \\
                                                                 \end{array}
                                                               \right)^{1/p}.\end{equation}
In fact, we can replace $t^{1/p}$ by any $C^1-$function $f_p$ which equals 0 in a neighborhood of 0 and equals $t^{1/p}$ in the vicinity of $\alpha^+.$ Note that $[f_p,\alpha]$ (as defined in \eqref{flambda}) then equals $[(\cdot)^{1/p},\alpha]_1$ (as defined above \eqref{flambda2}), except for the fourth quadrant. It follows by Theorem \ref{thmdk} that \eqref{lo9}, as a function of $(B,C)$, is Fr\'{e}chet differentiable and
\begin{equation*}\begin{aligned}&\left(
                                                                 \begin{array}{cc}
                                                                   \Lambda_{\alpha^+}+B & C \\
                                                                   C^* & C^*(\Lambda_{\alpha^+}+B)^{-1}C \\
                                                                 \end{array}
                                                               \right)^{1/p}=\left(
                                                                 \begin{array}{cc}
                                                                   \Lambda_{\alpha^+} & 0 \\
                                                                   0 & 0 \\
                                                                 \end{array}
                                                               \right)^{1/p}+[f_p,\alpha]\circ \left(
                                                                 \begin{array}{cc}
                                                                   B & C \\
                                                                   C^* & C^*(\Lambda_{\alpha^+}+B)^{-1}C \\
                                                                 \end{array}                                                         \right)\\&+O(\|E\|^2)=\Lambda_{\alpha}^{1/p}+[(\cdot)^{1/p},\alpha]_1\circ \left(
                                                                 \begin{array}{cc}
                                                                   B & C \\
                                                                   C^* & 0 \\
                                                                 \end{array}                                                         \right)+O(\|E\|^2)\end{aligned}\end{equation*}
where $O(\|E\|^2)\leq O(\|E\|^r)$ since $r<2$. For simplicity, the local independence of the ordo terms on $\alpha^+$ is left implicit in the remainder.
Comparing this with \eqref{lo}, we see that it is sufficient to show that
\begin{equation}\label{lo2}\begin{aligned}&\left(
                                                                 \begin{array}{cc}
                                                                   \Lambda_{\alpha^+}+B & C \\
                                                                   C^* & C^*(\Lambda_{\alpha^+}+B)^{-1}C+D \\
                                                                 \end{array}
                                                               \right)^{1/p}=\\&\left(
                                                                 \begin{array}{cc}
                                                                   \Lambda_{\alpha^+}+B & C \\
                                                                   C^* & C^*(\Lambda_{\alpha^+}+B)^{-1}C \\
                                                                 \end{array}
                                                               \right)^{1/p}+\left(
                                                                 \begin{array}{cc}
                                                                   0 & 0 \\
                                                                   0 & D^{1/p} \\
                                                                 \end{array}
                                                               \right)+O(\|E\|^r).\end{aligned}\end{equation}
It is clearly no restriction to assume that the basis has been chosen so that $D$ is diagonal; $D=\Lambda_\delta$. We let $\Delta$ be the vector obtained by adding zeros to $\delta$ so that $\Lambda_\Delta=
\left(
\begin{array}{cc}
0 & 0 \\
0 & \Lambda_\delta \\
\end{array}
\right)$. Set $\tilde E=\left(
\begin{array}{cc}
B & C \\
C^* & C^*(\Lambda_{\alpha^+}+B)^{-1}C \\
\end{array}
\right)$ and set $A_1=\Lambda_{\alpha}+\tilde E$. We denote the eigenvalues of $A_1$ by $\alpha_1$ and apply the spectral theorem to get a unitary matrix $ V$ such that $$\left(
                                                                 \begin{array}{cc}
                                                                   \Lambda_{\alpha^+}+B & C \\
                                                                   C^* & C^*(\Lambda_{\alpha^+}+B)^{-1}C \\
                                                                 \end{array}
                                                               \right)= V\left(
                                                                 \begin{array}{cc}
                                                                   \Lambda_{\alpha_1^+} & 0 \\
                                                                   0 & 0 \\
                                                                 \end{array}
                                                               \right) V^*$$
where ${\alpha_1^+}$ is the positive part of $\alpha_1$. We then have
\begin{equation*} V^*\left(
                                                                 \begin{array}{cc}
                                                                   \Lambda_{\alpha^+}+B & C \\
                                                                   C^* & C^*(\Lambda_{\alpha^+}+B)^{-1}C+\Lambda_\delta \\
                                                                 \end{array}
                                                               \right)^{1/p} V=
\left(
                                                                 \begin{array}{cc}
                                                                   \Lambda_{\alpha_1^+}+B_1 & C_1 \\
                                                                   C_1^* &  D_1 \\
                                                                 \end{array}
                                                               \right)^{1/p}\end{equation*}
where $B_1,C_1,D_1$ represent $ V^*\Lambda_\Delta  V$. More precisely, if we denote by $ V_1$ the matrix with the $l$ first columns of ${V}$ and by $ V_0$ the matrix with the $m$ remaining columns, we have $P_1( E)= V_1 V_1^*$ and $P_0( \tilde E)= V_0 V_0^*$ whereby it follows that e.g. $B_1= V_1^* \Lambda_\Delta  V_1$ is a matrix representation of  $P_1(\tilde E) \Lambda_\Delta P_1(\tilde E)$ on the subspace $\Ran P_1(\tilde E)$. Similarly, $C_1= V_1^* \Lambda_\delta  V_0$ corresponds to $P_1(\tilde E) \Lambda_\Delta P_0(\tilde E)$ and $D_1$ to $P_0(\tilde E) \Lambda_\Delta P_0(\tilde E)$. In particular $D_1\geq 0$. Clearly $\|\Lambda_\delta\|$ and $\|\tilde E\|$ are $O(\|E\|)$ so by Lemma \ref{gt2}, applied with $Z=\Lambda_{\Delta}$ and $b$ some suitable constant times $\|E\|$, we conclude that $B_1=O(\|E\|^3)$, $C_1=O(\|E\|^2)$ and $D_1=O(\|E\|)$, where the last follows from Lemma \ref{gt1}. This lemma, expressed in the new basis, also gives that $$\left(
                                                                 \begin{array}{cc}
                                                                   B_1 & C_1 \\
                                                                   C_1^* & D_1 \\
                                                                 \end{array}
                                                               \right)^{1/p}=\left(
                                                                 \begin{array}{cc}
                                                                   0 & 0 \\
                                                                   0 & D_1^{1/p} \\
                                                                 \end{array}
                                                               \right)+O(\|E\|^r).$$ Multiplying by $V^*$ from the left and $ V$ from the right, the identity \eqref{lo2} is seen to be equivalent with
\begin{equation}\label{lo1}\begin{aligned}&\left(
                                                                 \begin{array}{cc}
                                                                   \Lambda_{\alpha_1^+}+B_1 & C_1 \\
                                                                   C_1^* & D_1 \\
                                                                 \end{array}
                                                               \right)^{1/p}=\left(
                                                                 \begin{array}{cc}
                                                                   \Lambda_{\alpha_1^+} & 0 \\
                                                                   0 & 0 \\
                                                                 \end{array}
                                                               \right)^{1/p}+\left(
                                                                 \begin{array}{cc}
                                                                   0 & 0 \\
                                                                   0 &  D_1^{1/p} \\
                                                                 \end{array}
                                                               \right)+O(\|E\|^r).\end{aligned}\end{equation} As usual, we can assume that $V$ was chosen so that $D_1$ is diagonal; $D_1=\Lambda_{\delta_1}$. To complete the proof it suffices to prove \eqref{lo1} under this assumption.

We now introduce \begin{equation}\label{eigv5}V_{ap}=\left(
                                                                 \begin{array}{cc}
                                                                   I & -\Lambda_{\alpha_1^+}^{-1} C_1 \\
                                                                   C_1^*\Lambda_{\alpha_1^+}^{-1}  & I \\
                                                                 \end{array}
                                                               \right).\end{equation} and note that \begin{equation}\label{spec_dec_pert3}\left(
                                                                 \begin{array}{cc}
                                                                   \Lambda_{\alpha_1^+}+B_1 & C_1 \\
                                                                   C_1^* & \Lambda_{\delta_1} \\
                                                                 \end{array}
                                                               \right)= V_{ap}\left(
                                                                 \begin{array}{cc}
                                                                   \Lambda_{\alpha_1^+} & 0 \\
                                                                   0 & \Lambda_{\delta_1} \\
                                                                 \end{array}
                                                               \right)V_{ap}^*+O(\|E\|^3).\end{equation}
As in Lemma \ref{gt1} we can apply the Gram-Schmidt process to $V_{ap}$. The scalar product of any two columns in the first part of the matrix (i.e. quadrant 1 and 3, lets denote it $V_{ap,1}$) is $O(\|E\|^4)$, and the scalar product between columns from the first and the second part (quadrant 2 and 4, lets denote it $V_{ap,0}$) have scalar products which are $O(\|E\|^2)$. It thus follows that there is a unitary matrix $W$ with $V_{ap,1}=W_1+O(\|E\|^4)$ and $V_{ap,0}=W_0+O(\|E\|^2)$. Equality \eqref{spec_dec_pert3} thus implies that \begin{align*}&\left(
                                                                 \begin{array}{cc}
                                                                   \Lambda_{\alpha_1^+}+B_1 & C_1 \\
                                                                   C_1^* & \Lambda_{\delta_1} \\
                                                                 \end{array}
                                                               \right)= W_1\Lambda_{\alpha_1^+}W_1^*+O(\|E\|^4)+W_0\Lambda_{\delta_1}W_0^*+O(\|E\|^3)=\\& W\left(
                                                                 \begin{array}{cc}
                                                                   \Lambda_{\alpha_1^+} & 0 \\
                                                                   0 & \Lambda_{\delta_1} \\
                                                                 \end{array}
                                                               \right)W^*+O(\|E\|^3)\end{align*}
which in combination with Wihler's inequality \eqref{lipfrac} yields that \begin{equation*}\begin{aligned}&\left(
                                                                 \begin{array}{cc}
                                                                   \Lambda_{\alpha_1^+}+B_1 & C_1 \\
                                                                   C_1^* & \Lambda_{\delta_1} \\
                                                                 \end{array}
                                                               \right)^{1/p}=W\left(
                                                                 \begin{array}{cc}
                                                                   \Lambda_{\alpha_1^+} & 0 \\
                                                                   0 & \Lambda_{\delta_1} \\
                                                                 \end{array}
                                                               \right)^{1/p}W^*+O(\|E\|^{3/p})=\\&\left(
                                                                 \begin{array}{cc}
                                                                   \Lambda_{\alpha_1^+} & 0 \\
                                                                   0 & \Lambda_{\delta_1} \\
                                                                 \end{array}
                                                               \right)^{1/p}+O(\|E\|^{2})+O(\|E\|^{3/p})=\left(
                                                                 \begin{array}{cc}
                                                                   \Lambda_{\alpha_1^+}^{1/p} & 0 \\
                                                                   0 & \Lambda_{\delta_1}^{1/p} \\
                                                                 \end{array}
                                                               \right)+O(\|E\|^r)\end{aligned}\end{equation*} as $r\leq \min (2,3/p)$, which is \eqref{lo1} and the proof is complete.

\end{proof}

\textit{Remarks:} 1. Obviously, it is an interesting open problem how a first order formula would look like for $p\geq 3$. The expression \eqref{lo} is of course applicable for $p\geq 3$, so maybe all that is needed is a more careful analysis of the size of the error term, we leave this as an open problem.

2. If instead we are interested in $A^s$ where $s>1$, then the Daleckii-Krein theorem applies upon defining $t^s\equiv 0$ for $t<0$. Formula \eqref{lo} then applies upon switching $1/p$ with $s$ and $[(\cdot)^{1/p},\alpha]_1$ for $[(\cdot)^{s},\alpha]$.

%

\section{Perturbation of the matrix-modulus}\label{sec3}


In this final section we consider the perturbation of $|X|=\sqrt{X^*X}$, where $X$ is no longer confined to be Hermitian. There is plenty of prior work devoted to understanding the perturbation theory for this operation, see Ch.~X.2 and X.6 of \cite{bhatia2013matrix} as well as the articles \cite{barrlund1990perturbation,bhatia1994first,cavaretta2003lipschitz,kato1973continuity,mathias1993perturbation}, but none contain a formula reminiscent of the one we will present in Theorem \ref{modulus}.

We first remark that the non-square case can be treated as a special case of the square case, by zero-padding, and hence we assume that $X$ is a square $n\times n-$matrix.
Consider $X+Z$ where $Z$ is the perturbation. Let $X$ have singular value decomposition equal to $X=U\Lambda_{\sigma} V^*,$ where $\sigma$ are the singular values. As before assume that there are $l$ non-zero ones and denote the corresponding vector of positive singular values by ${\sigma^+}$. The kernel of $X$ then has dimension $m=n-l$.
Since $X+Z=U(\Lambda_\sigma+U^*EV)V^*$, it is easy to see that it suffices to consider $\Lambda_{\sigma}+\check Z$ where $\check Z=
U^*EV$. Furthermore, if $\Lambda_{\sigma}$ has a kernel, we divide $\check Z$ into the subblocks \begin{align*}
&\left(
\begin{array}{cc}
\check Z_{11} & \check Z_{12} \\
\check Z_{21}  & \check Z_{22} \\
\end{array}
\right)
\end{align*}
in analogy with previous sections. We introduce $\Xi_{\sigma^+}$ as the $l\times l$-matrix defined via
\begin{equation}\label{fsigma}(\Xi_{\sigma^+})_{i,j}=
  \frac{1}{\sigma_i+\sigma_j}
\end{equation} We then have
\begin{theorem}\label{modulus}
$$|X+Z|=|X|+V\left(
\begin{array}{cc}
\Xi_{\sigma^+}\circ(\Lambda_{\sigma^+}\check  Z_{11}+\check  Z_{11}^* \Lambda_{\sigma^+})& \check  Z_{12} \\
\check Z_{12}^*  & |\check Z_{22}| \\
\end{array}
\right)V^*+O(\|Z\|^{3/2}).$$
\end{theorem}
\textit{Remarks:} 1. It is a quite curious fact that $\check Z_{21}$ has no influence on the first order term.

2. In contrast to previous sections, $\check Z$ need not be self-adjoint even if $Z$ is. Indeed, if $X=U\Lambda_{\sigma} V^*$ is self-adjoint then $\alpha_j=\pm\sigma_j$ for all $j$. If $S$ is a diagonal matrix which is $\pm 1$ according to the sign of this identity, it is easy to see that $U=VS$. If we set $\hat Z=V^*Z V$ (which thus is self-adjoint) we then get $$\check Z=S\hat Z.$$
This observation also shows that if $X$ has no negative eigenvalues, then the formula simplifies to $$|X+Z|=X+V\left(
\begin{array}{cc}
\hat  Z_{11} & \hat  Z_{12} \\
\hat Z_{21}  & |\hat Z_{22}| \\
\end{array}
\right)V^*+O(\|Z\|^{3/2}).$$

3. Now suppose that $X,Z\in\H_n$ and that $X$ is invertible. Introducing the matrix $\Xi_{\sigma}^\alpha$ by $$\Xi^\alpha_\sigma(i,j)=\frac{\alpha_i+\alpha_j}{\sigma_i+\sigma_j},$$
the formula simplifies to $$|X+Z|=|X|+V
\Big(\Xi_{\sigma}^\alpha\circ\hat  Z \Big)V^*+O(\|Z\|^{3/2}),$$ which follows by noting that $\hat Z_{11}=\hat Z$, $ \Lambda_{\sigma^+} S= \Lambda_{\sigma} S=\Lambda_{\alpha}$ so $\Xi_{\sigma^+}\circ(\Lambda_{\sigma^+} \check  Z_{11}+\check  Z_{11}  \Lambda_{\sigma^+})=\Xi_{\sigma^+}\circ(\Lambda_{\alpha}\hat  Z+\hat  Z \Lambda_{\alpha})$. In particular $\Xi_{\sigma}^\alpha$ is a matrix whose elements are less than 1 in modulus, since $|\alpha_i|=\sigma_i$. However, this formula can be deduced from the Daleckii-Krein theorem without use of Theorem \ref{modulus}, so it is unlikely that this observation leads to something new. It is however curious that this beautiful formula is not found explicitly anywhere in the literature.

4. For a long time it was an open problem whether $X\mapsto |X|$ was Lipschitz-continuous with respect to the operator norm (in the infinite-dimensional case), which was finally resolved in the negative by T.~Ando \cite{kato1973continuity} (with help from A.~McIntosh \cite{mcintosh1971counterexample}), even if both $X$ and $Z$ are self-adjoint. Ando's counterexample can thus be understood as a consequence of that Hadamard-multiplication with $\Xi_\sigma^\alpha$ is not bounded independently of the dimension of the underlying space.

\begin{proof}
We have \begin{equation}\label{p}(X+Z)^*(X+Z)=V(\Lambda_{\sigma}^2+\Lambda_{\sigma} \check Z+\check Z^*\Lambda_{\sigma}+\check Z^*\check Z)V^*.\end{equation}
Setting $A=\Lambda_{\sigma}^2,$ we have $\alpha=\sigma^2$  so $\Lambda_{ \alpha^+}=\Lambda_{\sigma^+}^2$
and, expanding \eqref{p} in the four corresponding subblocks, we get $$V\left(\begin{array}{cc}
 \Lambda_{\sigma^+}^2+\Lambda_{\sigma^+}\check Z_{11}+\check Z_{11}\Lambda_{\sigma^+}+O(\| Z\|^2) & \Lambda_{\sigma^+}\check Z_{12}+O(\| Z\|^2) \\
(\Lambda_{\sigma^+}\check Z_{12}+O(\| Z\|^2))^*  & \check Z_{12}^*\check Z_{12}+  \check Z_{22}^*\check Z_{22}+O(\|Z\|^3)\\
\end{array}
\right)V^*$$ We let ${B}$ be the perturbation in the first quadrant, i.e. the term denoted $\Lambda_{\sigma^+}\check Z_{11}+\check Z_{11}\Lambda_{\sigma^+}+O(\| Z\|^2)$ above, and similarly we set $ C=\Lambda_{\sigma^+}\check Z_{12}+O(\| Z\|^2).$ This gives that $C^*(\Lambda_{\sigma^+}^2+B)^{-1}C$ simply becomes $\check Z_{12}^*\check Z_{12}+O(\|Z\|^3)$, and therefore we set $ D=\check Z_{22}^*\check Z_{22}$ (no error term).
With this notation, \eqref{p} can be expressed as $$V\left(\begin{array}{cc}
 \Lambda_{\sigma^+}^2+ B & C \\
 C^*  &  C^*(\Lambda_{\sigma^+}^2+ B)^{-1} C+ D+O(\|Z\|^3) \\
\end{array}
\right)V^*$$ and by Wihler's inequality \eqref{lipfrac} we get that $\sqrt{(X^*+Z^*)(X+Z)}$ can be expressed as \begin{equation}\label{o0}V\sqrt{\left(
\begin{array}{cc}
\Lambda_{\sigma^+}^2 & 0 \\
0  & 0 \\
\end{array}
\right)+\left(\begin{array}{cc}
 B &  C \\
 C^*  &  C^*(\Lambda_{\sigma^+}^2+ B)^{-1} C+{D} \\
\end{array}
\right)}V^*+O(\|Z\|^{3/2}).\end{equation}
We shall now apply Theorem \ref{thm_frac_pow} (with $p=2$) to $ A=\Lambda^2_{\sigma}$ and the Hermitian perturbation $ E=\left(\begin{array}{cc}
 B &  C \\
 C^*  &  C^*(\Lambda_{\sigma^+}^2+ B)^{-1} C+{ D} \\
\end{array}
\right)$ (where it is not necessary to put a hat on $E$ since $ A$ is already diagonal). It is easy to see that \eqref{flambda2} applied to $\sqrt{\cdot}$ and $\alpha=\sigma^2$ yields the $n\times n$-matrix defined via
\begin{equation*}(\Xi_{\sigma})_{i,j}=([\sqrt{\cdot},\sigma^2]_1)_{i,j}=
\left\{\begin{array}{ll}
  \frac{1}{\sigma_i+\sigma_j} & \sigma_i>0\text{ or } \sigma_j>0 \\
  1 & \sigma_i=\sigma_j=0 \\
\end{array}\right.
\end{equation*}
where $\Xi_{\sigma^+}$ is the $l\times l$ upper right corner of this matrix. Hence
\begin{align*}&\sqrt{\left(
\begin{array}{cc}
\Lambda_{\sigma^+}^2 & 0 \\
0  & 0 \\
\end{array}
\right)+\left(\begin{array}{cc}
 B &  C \\
 C^*  &  C^*(\Lambda_{\sigma^+}^2+ B)^{-1} C+{ D} \\
\end{array}
\right)}=\\&\sqrt{\left(
\begin{array}{cc}
\Lambda_{\sigma^+}^2 & 0 \\
0  & 0 \\
\end{array}
\right)}+\Xi_\sigma\circ
\left(\begin{array}{cc}
 B &  C \\
 C^*  & \sqrt{D} \\
\end{array}
\right)+O(\| E\|^{3/2})\\&\left(
\begin{array}{cc}
\Lambda_{\sigma^+} & 0 \\
0  & 0 \\
\end{array}
\right)+\Xi_\sigma\circ
\left(\begin{array}{cc}
\Lambda_{\sigma^+}\check E_{11}+\check E_{11}^*\Lambda_{\sigma^+} & \Lambda_{\sigma^+} \check E_{12} \\
\check E_{12}^*\Lambda_{\sigma^+}  & |E_{22}| \\
\end{array}
\right)+O(\|Z\|^2)+O(\| E\|^{3/2})=\\&\left(
\begin{array}{cc}
\Lambda_{\sigma^+} & 0 \\
0  & 0 \\
\end{array}
\right)+\left(
\begin{array}{cc}
\Xi_{\sigma^+}\circ(\Lambda_{\sigma^+}\check  E_{11}+\check  E_{11}^* \Lambda_{\sigma^+})& \check  E_{12} \\
\check E_{12}^*  & |\check E_{22}| \\
\end{array}
\right)+O(\| Z\|^{2})+O(\|E\|^{3/2}).\end{align*}
Clearly $O\| E\|=O(\|Z\|)$. Inserting this into \eqref{o0}, the theorem follows.
\end{proof}


\bibliographystyle{plain}
\bibliography{MCref}

\begin{thebibliography}{10}

\bibitem{barrlund1990perturbation}
Anders Barrlund.
\newblock Perturbation bounds on the polar decomposition.
\newblock {\em BIT Numerical Mathematics}, 30(1):101--113, 1990.

\bibitem{bhatia1994first}
Rajendra Bhatia.
\newblock First and second order perturbation bounds for the operator absolute
  value.
\newblock {\em Linear algebra and its applications}, 208:367--376, 1994.

\bibitem{bhatia2013matrix}
Rajendra Bhatia.
\newblock {\em Matrix analysis}, volume 169.
\newblock Springer Science \& Business Media, 2013.

\bibitem{boyd2004convex}
Stephen Boyd and Lieven Vandenberghe.
\newblock {\em Convex optimization}.
\newblock Cambridge university press, 2004.

\bibitem{carlsson2018perturbation3}
Marcus Carlsson.
\newblock A new functional calculus for hermitian matrices based on
  vector-fields, and its perturbation theory.
\newblock {\em To appear}.

\bibitem{carlsson2018perturbation1}
Marcus Carlsson.
\newblock Perturbation theory for the spectral decomposition of hermitian
  matrices.
\newblock {\em arXiv preprint arXiv:1809.09480}, 2018.

\bibitem{cavaretta2003lipschitz}
Alfred~S Cavaretta and Laura Smithies.
\newblock Lipschitz-type bounds for the map $a\mapsto| a|$ on $l (h)$.
\newblock {\em Linear algebra and its applications}, 360:231--235, 2003.

\bibitem{cayley1858ii}
Arthur Cayley.
\newblock A memoir on the theory of matrices.
\newblock {\em Philosophical transactions of the Royal society of London},
  148:17--37, 1858.

\bibitem{daletskii1965integration}
Ju~L Daletskii and SG~Krein.
\newblock Integration and differentiation of functions of hermitian operators
  and applications to the theory of perturbations.
\newblock {\em AMS Translations (2)}, 47(1-30), 1965.

\bibitem{higham2008functions}
Nicholas~J Higham.
\newblock {\em Functions of matrices: theory and computation}, volume 104.
\newblock Siam, 2008.

\bibitem{kato1973continuity}
Tosio Kato.
\newblock Continuity of the map s?| s| for linear operators.
\newblock {\em Proceedings of the Japan Academy}, 49(3):157--160, 1973.

\bibitem{mathias1993approximation}
Roy Mathias.
\newblock Approximation of matrix-valued functions.
\newblock {\em SIAM journal on matrix analysis and applications},
  14(4):1061--1063, 1993.

\bibitem{mathias1993perturbation}
Roy Mathias.
\newblock Perturbation bounds for the polar decomposition.
\newblock {\em SIAM Journal on Matrix Analysis and Applications},
  14(2):588--597, 1993.

\bibitem{mcintosh1971counterexample}
Alan McIntosh.
\newblock Counterexample to a question on commutators.
\newblock {\em Proceedings of the American Mathematical Society},
  29(2):337--340, 1971.

\bibitem{sendov2007higher}
Hristo~S Sendov.
\newblock The higher-order derivatives of spectral functions.
\newblock {\em Linear algebra and its applications}, 424(1):240--281, 2007.

\bibitem{wihler2009holder}
Thomas~P Wihler.
\newblock On the h{\"o}lder continuity of matrix functions for normal matrices.
\newblock {\em Journal of inequalities in pure and applied mathematics},
  10:1--5, 2009.

\end{thebibliography}
\end{document}